\documentclass[11pt]{amsart}
\usepackage[margin=1in]{geometry}

\usepackage{amsthm,amsmath,amssymb,color}

\usepackage[hidelinks]{hyperref}
\usepackage{graphicx}

\usepackage{pdflscape}
\usepackage{afterpage}
\usepackage{comment}
\usepackage[foot]{amsaddr}
\usepackage{threeparttable}

\theoremstyle{definition}
\newtheorem{defi}{Definition}[section]

\theoremstyle{plain}
\newtheorem{thm}[defi]{Theorem}
\newtheorem{lemma}[defi]{Lemma}

\newtheorem{rem}{Remark}
\newcommand{\lr}[1]{\lbrace #1 \rbrace}
\newcommand{\lrr}[1]{\langle #1 \rangle}
\newcommand{\vvec}[1]{\overrightarrow{#1}}

\makeatletter
\def\imod#1{\allowbreak\mkern10mu({\operator@font mod}\,\,#1)} 
\def\@setcopyright{}                                           
\def\serieslogo@{}
\makeatother

\usepackage[pagewise]{lineno}

\begin{document}


\author[J.M.P.~Balmaceda]{Jose Maria P. ~Balmaceda}
\address[J.M.P.~Balmaceda]{Institute of Mathematics, University of the Philippines Diliman, 1101 Quezon City, Philippines}
\email{jpbalmaceda@up.edu.ph}

\author[D.V.A.~Briones]{Dom Vito A. ~Briones}
\address[D.V.A.~Briones, Corresponding author]{Institute of Mathematics, University of the Philippines Diliman, 1101 Quezon City, Philippines}
\email{dabriones@up.edu.ph}

\title{Families of Association Schemes on Triples from Two-Transitive Groups}

\begin{abstract} 
Association schemes on triples (ASTs) are ternary analogues of classical association schemes. Similar to how Schurian association schemes arise from transitive groups, ASTs arise from 
two-transitive groups. 
In this paper, we obtain the third valencies and the number of relations of the ASTs obtained from 
two-transitive permutation groups. 
Further, we obtain the intersection numbers for the ASTs obtained from $P\Gamma L(k,n)$, $PSL(2,n)$, $A \Gamma L(k,n)$, and the sporadic two-transitive groups. In particular, the ASTs from $P\Gamma L(k,n)$, $PSL(2,n)$, and the sporadic groups are commutative.
\end{abstract}



\date{\today}

\maketitle

\section[Introduction]{Introduction}

 An association scheme is a set $X=\lr{R_i}_{i=0}^m$ of binary relations on a finite nonempty set $\Omega$ that partitions $\Omega \times \Omega$ and which satisfies certain symmetry requirements. These requirements are flexible enough to accommodate a variety of mathematical objects yet rigid enough to afford desirable algebraic and combinatorial properties. For instance, the adjacency algebras of association schemes and their duals under the Hadamard product are semisimple, and the parameters of these algebras may be used in the classification of certain types of graphs \cite{bannai_algebraic_1984}.

In the 1990 paper \cite{mesner_association_1990}, Mesner and Bhattacharya extended the notion of association schemes to association schemes on triples (ASTs). Here, the underlying relations are ternary instead of binary. By generalizing the usual $m\times m$ binary square matrix product $(A,B)\mapsto AB$ to an $m\times m \times m$ ternary cubic hypermatrix product $(A,B,C)\mapsto ABC $, the resulting adjacency hypermatrices form a 
non-associative ternary algebra. In the same paper, the authors showed that a two-transitive action of a group $G$ on a set $\Omega$ induces an AST, denoted $X=\lr{R_i}_{i=0}^m$, by letting the relations $
R_i$ in $X$ be the orbits of the induced action of $G$ on $\Omega \times \Omega \times \Omega$. 
In particular, the authors gave the sizes (number of relations) of the ASTs obtained from the affine general linear groups $AGL(1,n)$, the projective special linear groups $PSL(2,n)$, the Suzuki groups $Sz(2^{2k+1})$, and the Higman-Sims group $HS$. 
Moreover, they also determined some intersection numbers of the ASTs obtained from $AGL(1,n)$ and $PSL(2,n)$ by computing for some products of the adjacency hypermatrices of the corresponding ASTs.  

Other works on ASTs include generalizations of identity elements and inverse elements for the ternary algebras of ASTs \cite{mesner_ternary_1994}, ASTs with relations invariant under a transitive cyclic subgroup of the symmetric group \cite{Zealand2021}, an algorithm for classifying ASTs whose relations are invariant under a given group action \cite{balma2022}, and ASTs obtained by decomposing or taking the unions of relations of an existing AST \cite{balmasurvey}. However, much remains unknown for ASTs. For instance, it is not known whether the ternary algebras of ASTs satisfy analogues for the semisimplicity and duality properties of the adjacency algebras of classical association schemes.

To further the study of ASTs, we extend the work in \cite{mesner_association_1990} by obtaining the sizes and third valencies of the ASTs obtained from the two-transitive permutation groups. These are given in Theorem \ref{main1}.

\begin{thm}\label{main1}
 The third valencies and number of relations of the ASTs obtained from the symmetric groups $S_k$, the projective semilinear groups $P\Gamma L(k,n)$, the projective special linear groups $PSL(2,n)$, the Suzuki groups $Sz(2^{2k+1})$, the Ree groups $Ree(3^{2k+1})$, the affine semilinear groups $A\Gamma L(k,n)$, the projective unitary groups $PGU(3,n)$, the symplectic groups $Sp(2k,2)$, and the sporadic two-transitive groups are given in Table \ref{tab:thirdvalencies_and_sizes}.  
\end{thm}

Theorem \ref{main1} is obtained by determining the orbits of the groups' two-point stabilizers. Indeed, the number of nontrivial relations of an AST obtained from a two-transitive group is equal to the number of orbits of a two-point stabilizer while the third valencies are the sizes of these orbits. 

Furthermore, by determining which elements of the underlying space can be sent to one another through certain elements of the acting group, we extend some of the results in \cite{mesner_association_1990} regarding the intersection numbers of ASTs obtained from $PSL(2,n)$ and $AGL(1,n)$ to the intersection numbers of ASTs obtained from $P\Gamma L(k,n)$ and $A \Gamma L(k,n)$. Additionally, we complete the intersection numbers of ASTs from $PSL(2,n)$ which were partially obtained in \cite{mesner_association_1990}. Through GAP 4.11.1 \cite{GAP4}, we also determine the intersection numbers of the ASTs obtained from the sporadic two-transitive groups. These intersection numbers correct some errors in \cite{mesner_association_1990} and are given in Theorem \ref{main2}.

\begin{thm}\label{main2}
    The intersection numbers of the ASTs obtained from the affine semilinear groups $A\Gamma L(k,n)$, the projective groups $P\Gamma L(k,n)$ and $PSL(2,n)$, and the sporadic two-transitive groups are given in Tables \ref{tab:pgaml} to \ref{tab:CO}.
\end{thm}


\section{Preliminaries}
This section is based mostly on \cite{mesner_association_1990} and \cite{dixon1996permutation}. We define association schemes on triples (ASTs), state relevant properties of ASTs, and describe relationships between ASTs and two-transitive groups. 

\subsection{Association Schemes on Triples}

Similar to a classical association scheme, an AST on a set $\Omega$ is a set of ternary relations $X=\lr{R_i}_{i=0}^m$ that partitions $\Omega \times \Omega \times \Omega$ and which satisfies certain symmetry requirements. More precisely, we have the following definition.

\begin{defi}\label{def_AST}
Let $\Omega$ be a finite set with at least 3 elements. An association scheme on triples (AST) on $\Omega$ is a partition $X=\lr{R_i}_{i=0}^m$ of $\Omega \times \Omega \times \Omega$ with $m\geq 4$ such that the following conditions hold.

\begin{enumerate}
    \item  For each $i\in \lr{0,\ldots,m}$, there exists an integer $n_i^{(3)}$ such that for each pair of distinct $x,y\in \Omega$, the number of $z\in \Omega$ with $(x,y,z)\in R_i$ is $n_i^{(3)}$.
    \item (Principal Regularity Condition.) For any $i,j,k,l \in \lr{0,\ldots,m}$, there exists a constant $p_{ijk}^l$ such that for any $(x,y,z)\in R_l$, the number of $w$ such that $(w,y,z)\in R_i$, $(x,w,z)\in R_j$, and $(x,y,w)\in R_k$ is $p_{ijk}^l$.
    \item For any $i\in \lr{0,\ldots, m}$ and any $\sigma \in S_3$, there exists a $j\in \lr{0,\ldots,m}$ such that \[R_j= \lr{(x_{\sigma(1)},x_{\sigma(2)},x_{\sigma(3)}):(x_1,x_2,x_3)\in R_i}.\]
    \item The first four relations are $R_0=\lr{(x,x,x): x\in \Omega}$, $R_1=\lr{(x,y,y):x,y\in \Omega,x\neq y}$, $R_2=\lr{(y,x,y):x,y\in \Omega, x\neq y}$, and $R_3=\lr{(y,y,x):x,y\in \Omega, x\neq y}$.
\end{enumerate}
\end{defi}
Analogous to the valency of a classical association scheme, we name $n_i^{(3)}$ the third valency of $R_i$. By conditions 1 and 3, there exist for each $i$ the constants $n_i^{(1)}=\vert \lr{z\in \Omega:(z,x,y) \in R_i}\vert $ and $n_i^{(2)}=\vert \lr{z\in \Omega:(x,z,y) \in R_i}\vert $ independent of any pair of distinct elements $x,y\in \Omega$. We term $n_i^{(1)}$ and $n_i^{(2)}$ the first and second valency of $R_i$, respectively. We call $p_{ijk}^l$ the intersection number corresponding to $R_i$, $R_j$, $R_k$, and $R_l$. The relations $R_0$, $R_1$, $R_2$, and $R_3$ are called the trivial relations. The other relations are the nontrivial relations.  

Similar to classical association schemes, we can represent an AST in terms of hypermatrices. 
Let $\nu=\vert \Omega\vert $. To each $R_i\in X$ we associate the $\nu \times \nu \times \nu$ cubic hypermatrix $A_i$ whose entries are indexed by $\Omega$. For $i\in\{0,\ldots,m\}$, define $A_i$ by \[{(A_i)}_{xyz}=\begin{cases} 1, &\text{if }(x,y,z)\in R_i, \\ 0, & \text{otherwise.}\end{cases}\] Here, $(A_i)_{xyz}$ is the $(x,y,z)$-entry of $A_i$. To emphasize their particular nature, the hypermatrices $A_0, A_1, A_2$, and $A_3$ corresponding to the trivial relations $R_0, R_1, R_2$, and $R_3$ may also be denoted by $I_0, I_1, I_2$, and $I_3$, respectively.

The set of $\nu \times \nu \times \nu$ hypermatrices with complex entries indexed by $\Omega$ forms a complex vector space under componentwise addition and scalar multiplication. To provide it with a ternary algebra structure, we have the following ternary operation.

\begin{defi} \label{def_ternprod}
Let $A$, $B$, and $C$ be $\nu \times \nu \times \nu$ hypermatrices with complex entries indexed by $\Omega$. The ternary product $D=ABC$ is the $\nu \times \nu \times \nu$ hypermatrix given by \[(D)_{xyz}=\sum_{w\in \Omega}(A)_{wyz} (B)_{xwz} (C)_{xyw}.\]
\end{defi}

This ternary operation is multilinear and generally not associative \cite{mesner_association_1990} (in the sense of Lister in \cite{lister1971ternary}). With this ternary operation, the vector space generated by the hypermatrices $\lr{A_i}_{i=0}^m$ of an AST $X=\lr{R_i}_{i=0}^m$ becomes a ternary algebra. The structure constants of this ternary algebra 
are the intersection numbers $p_{ijk}^l$, as stated in the following theorem.

\begin{thm}[Theorem 1.4, \cite{mesner_association_1990}]\label{thm_adjhyperprod}
Let $X=\lr{R_i}_{i=0}^m$ be an AST on a set $\Omega$ and $\lr{A_i}_{i=0}^m$ be the corresponding adjacency hypermatrices. Then for any $i,j,k\in \lr{0,\ldots,m}$, one has \[A_i A_j A_k = \sum_{l=0}^m p_{ijk}^l A_l.\]
\end{thm}

This theorem allows us to state results regarding the intersection numbers of ASTs via expressing them as equations in the adjacency hypermatrices. In relation to this, we state the following theorem from \cite{mesner_association_1990}, which tells us that the adjacency matrices of the nontrivial relations generate a ternary subalgebra of the adjacency algebra. In particular, $p_{ijk}^l =0$ when $i,j,k>3$, and $l\leq 3$.

\begin{thm}[Corollary 2.8, \cite{mesner_association_1990}] \label{thm_subalge}
In an AST with relations $\lr{R_i}_{i=0}^m$, the adjacency hypermatrices $\lr{A_i}_{i=4}^m$ of the nontrivial relations generate a ternary subalgebra of the ternary algebra generated by the adjacency hypermatrices $\lr{A_i}_{i=0}^m$.
\end{thm}

According to \cite{mesner_association_1990}, the most interesting intersection numbers are those that arise from the subalgebra in Theorem \ref{thm_subalge}. Indeed, other theorems and remarks in \cite{mesner_association_1990} provide values and restrictions for $p_{ijk}^l$ when $\lr{i,j,k,l} \cap \lr{0,1,2,3} \neq \varnothing$. For instance, we state without proof the following consequences of Proposition 2.7 of \cite{mesner_association_1990}.

\begin{lemma}\label{ex_int5}\label{rem_trivial_prod}
    
Let $X=\lr{R_i}_{i=0}^m$ be an AST on a set $\Omega$. The following statements hold.
\begin{enumerate}
    \item Let $R_i$ be a trivial relation, $R_j$ and $R_k$ be nontrivial relations, and $R_l$ be any relation. Then the only intersection numbers of the form $p_{ijk}^l$, $p_{jik}^l$, or $p_{jki}^l$ that may be nonzero are $p_{1jk}^1$, $p_{j2k}^2$, and $p_{jk3}^3$. 
    \item If $m=4$, then $A_4 A_4 A_4=(\vert \Omega\vert -3)A_4$, $I_1 A_4 A_4 = (\vert \Omega\vert -2) I_1$, $A_4 I_2 A_4 = (\vert \Omega\vert -2) I_2$, and $A_4 A_4 I_3 = (\vert \Omega\vert -2) I_3$. Moreover, $n_4^{(3)}=\vert \Omega\vert-2.$
\end{enumerate}

\end{lemma}

\subsection{Two-Transitive Permutation Groups and ASTs}
Transitive group actions induce classical association schemes called Schurian association schemes. The theorem below states that any two-transitive group action analogously produces ASTs.

\begin{thm}[Theorem 4.1, \cite{mesner_association_1990}]\label{thm_ASTconstruct}
Let $G$ be a two-transitive group acting on a set $\Omega$. Then the orbits of the natural action of $G$ on $\Omega \times \Omega \times \Omega$ form the relations of an AST $X$ on $\Omega$. 

\end{thm}

Notice that $X$ has only one nontrivial relation if and only if all triples of pairwise distinct elements belong to the same orbit. This yields the following remark.

\begin{rem}\label{rem_oneNontrivial}
The AST $X$ obtained from a two-transitive action of a group $G$ on $\Omega$ has only one nontrivial relation if and only if $G$ acts three-transitively on $\Omega$. 
\end{rem}

Further, if $G$ is a group acting on a set $\Omega$ and $H$ is a subgroup of $G$, then the orbits of $G$ on $\Omega$ are unions of orbits of $H$. This bears the following consequence.

\begin{rem}\label{rem_orbit_implies_equality}
If $G$ is a group acting two-transitively on a set $\Omega$, and the restriction of the action to a subgroup $H$ of $G$ remains two-transitive on $\Omega$, then the ASTs obtained from $G$ and from $H$ are equal if and only if the ASTs have the same number of relations.

\end{rem}

We include here a lemma from \cite{mesner_association_1990} that is useful in determining the sizes of the ASTs obtained from two-transitive groups.  

\begin{lemma}[Lemma 4.2, \cite{mesner_association_1990}]\label{lem_mesner_orbit} 
Let $G$ be a two-transitive group acting on a set $\Omega$. Let $\Delta=[(x,y,z)]$ be an orbit of $G$ on $\Omega \times \Omega \times \Omega$ with $x,y$, and $z$ pairwise distinct. For $a,b \in \Omega$, $a\neq b$, let \[\Delta(a,b)=\lr{c:(a,b,c)\in \Delta}.\] Then $\Delta(a,b)$ is a $G_{a,b}$-orbit on $\Omega\setminus \lr{a,b}$. Furthermore, the map $\Delta \mapsto \Delta(a,b)$ is a bijection between the nontrivial $G$-orbits on $\Omega\times \Omega \times \Omega$ and the $G_{a,b}$-orbits on $\Omega \setminus \lr{a,b}$.
\end{lemma}

The following remark gives additional information obtained from Lemma \ref{lem_mesner_orbit} and the bijection $\Delta \mapsto \Delta(a,b)$.

\begin{rem} \label{rem_representative}
    Let $G$ be a group acting two-transitively on a set $\Omega$ and $X=\lr{R_i}_{i=0}^m$ be the induced AST.

    \begin{enumerate}
        \item For $a,b\in \Omega$ with $a\neq b$, the representatives of $R_i=[(a,b,c)]$ of the form $(a,b,d)$ for some $d\in \Omega$ are those triples with $d$ in the orbit of $c$ under the action of $G_{a,b}$ on $\Omega \setminus \lr{a,b}$. In particular, $n_i^{(3)}$ is the size of the orbit of $c$ under the action of $G_{a,b}$.  
        \item By varying the first coordinate or the second coordinate instead of the third coordinate, we obtain the values of the first and second valencies $n_i^{(1)}$ and $n_i^{(2)}$. Indeed, given nontrivial relations $R_i=[(a,b,c)]$, $R_j=[(a,c,b)]$, and $R_k=[(c,a,b)]$, the valencies $n_i^{(3)}=n_j^{(2)}=n_k^{(1)}$ are all equal to the size of the orbit of $c$ under the action of $G_{a,b}$.
    \end{enumerate}
\end{rem}

As a particular application of Lemma \ref{lem_mesner_orbit}, the authors of \cite{mesner_association_1990} determined that the size of the AST obtained from $PSL(2,n)$ is 5 if $n$ is even and $6$ if $n$ is odd. Further, they showed that the size of the AST obtained from $AGL(1,n)$ is precisely $n+2$ for $n$ a prime power larger than 2. They also determined intersection numbers of these ASTs in terms of their adjacency hypermatrices. 
Moreover, they determined that the size of the AST obtained from the Suzuki group $Sz(2^{2k+1})$ is $2^{2k+1}+5$ while the size of the AST obtained from the Higman-Sims group $HS$ is $7$. 

\section{ASTs from two-transitive groups}
We prove Theorems \ref{main1} and \ref{main2} in this section. A list of the groups considered and their explicit descriptions can be found in \cite{dixon1996permutation}. 
 For each group, we obtain a two-point stabilizer and the orbits of this stabilizer. Lemma \ref{lem_mesner_orbit} and Remark \ref{rem_representative} then yield the nontrivial relations and the third valencies of the corresponding ASTs. The third valencies of the nontrivial relations and the total number of relations of the ASTs are given in Table \ref{tab:thirdvalencies_and_sizes}.
 Meanwhile, by working with the orbits of the groups on the underlying Cartesian triple product, we obtain the intersection numbers of the ASTs listed in Theorem \ref{main2}. These are summarized in Tables \ref{tab:pgaml} to \ref{tab:CO}.
 The tables utilize definitions and notations from \cite{dixon1996permutation}, \cite{mesner_association_1990}, and the relevant subsections below. 

        
\begin{table}[htbp!]
    \centering
    \begin{tabular}{|c|c|c|}
    
    \hline 
    \textbf{Group}     & \textbf{Third Valencies} & \textbf{AST Size}\\ \hline
     $S_k$, $k\geq 3$   & $k-2$ & 5\\ \hline
     
     $P\Gamma L(k,n)$, $k\geq 3$      & $n-1$, $\frac{n^2(n^{k-2}-1)}{n-1}$   & 6\\ \hline
         
     $PSL(2,n)$, $n$ odd     & $\frac{n-1}{2}$ & 6\\ \hline
     
     $A\Gamma L(k,n)$, $k\geq 2$    &$\operatorname{deg}_{GF(p)}(a)$ where $a\in GF(n)$, $n^k-n$ & $  3 + \frac{1}{\alpha}\sum_{\beta=1}^\alpha p^{\gcd(\alpha,\beta)}$\\\hline

   $Sz(2^{2k+1})$, $k\geq 0$      &  $2^{2k+1}-1$  & $2^{2k+1}+5$\\\hline
          
   $Ree(3^{2k+1})$, $k\geq 0$      &   $3^{2k+1}-1$, $\frac{3^{2k+1}-1}{2}$   & $(3^{2k+1})^2+3^{2k+1}+6$\\ \hline
         
       $PGU(3,n)$  & $n^2-1$, $n-1$ &$n+5$\\ \hline

        $Sp^\varepsilon(2k,2)$, $k\geq 3$ &  $2^{2k-2}+\varepsilon2^{k-1}-2$, $2^{2k-2}$  & 6\\ \hline
        $PSL(2,11)$ (degree $11$)&3, 6 &6  \\ \hline
         $A_7$ (degree $15$) &1, 12&6 \\ \hline
         $HS$&12, 72, 90&7\\ \hline
           $Co_3$&112, 162 &6\\ \hline
    \end{tabular}

\medskip
    
    \caption{Third valencies and number of relations of ASTs from two-transitive groups}
    \label{tab:thirdvalencies_and_sizes}
\end{table}

\subsection{\texorpdfstring{${S_k}$}{Sk}}
Let $k\geq 3$. From \cite{dixon1996permutation}, it is known that $S_k$ is $k$-transitive. Remark \ref{rem_oneNontrivial} implies the associated AST has size $m=4$. Thus, Lemma \ref{ex_int5} describes the parameters of this AST. 

\begin{rem}
Remark \ref{rem_orbit_implies_equality} ensures that the AST obtained from any group acting three-transitively on a set $\Omega$ with $k\geq 3$ points is equal to the AST obtained from the permutation group $S_{\Omega}$ acting on $\Omega$. Such groups include $PGL(2,n)$ for $n$ a prime power, $PSL(2,n)$ for $n$ an even prime power, $A_k$ for $k\geq 5$, and the Mathieu groups. Accordingly, we refrain from tabulating the parameters of the ASTs from these groups.   
\end{rem}

\subsection{\texorpdfstring{${P\Gamma L(k,n)}$}{PGamL(k,n)}} Let $k\geq 3$ and $n$ be a power of a prime. For each homogeneous vector $y=[y_1:\cdots:y_k]^T\in PG(k-1,n)$, denote by $y'$ the affine vector $\frac{1}{y_j}(y_1,y_2,\ldots,y_k)^T$, where $j$ is the first coordinate such that $y_j\neq 0$. Further,
denote by $u$, $v$, $w$, and $x$ the respective elements  $[1:0:\cdots:0]^T, [0:1:0:\cdots:0]^T$, $[1:1:0\cdots:0]^T$, and $[0:0:1:0:\cdots:0]^T$ of $PG(k-1,n)$. 

Direct computation reveals that the two-point stabilizer $P\Gamma L(k,n)_{u,v}$ consists of the maps $z\mapsto A \phi(z)$ for some $A\in GL(k,n)$ and $\phi \in Gal(GF(n))$. Further computations yield the two orbits of $P\Gamma L(k,n)_{u,v}$, namely $\{[1:c:0:\cdots:0]^T:c\neq 0\}$ of size $n-1$ and $\{[c_1:\cdots:c_n]^T:(c_3,\ldots,c_k)\neq (0,\ldots,0)\}$ of size $\frac{n^2(n^{k-2}-1)}{n-1}$.

We now obtain the intersection numbers of the AST $X=\{R_i\}_{i=0}^5$ from $P\Gamma L(k,n)$. For clarity, let $R^w=[(u,v,w)]$ and $R^x=[(u,v,x)]$. Correspondingly, let $A^w$ and $A^x$ denote the respective adjacency hypermatrices. We first prove that $A^x A^x A^x = \frac{(n^{k-2}-1)n^2}{n-1}A^w + (\frac{n^k-1}{n-1}-3n)A^x$. 
Theorems \ref{thm_adjhyperprod} and \ref{thm_subalge} imply that it suffices to determine $p_{ijk}^l$ with $R_i=R_j=R_k=R^x$ and $R_l\in\{R^w,R^x\}$. For this, we utilize the principal regularity condition from Definition \ref{def_AST}.  
Initially, consider $R_l=R^w$ and $(u,v,w)\in R_l$. Remark \ref{rem_representative} implies that $(u,v,z)\in [(u,v,x)]$ is equivalent to $z=[x_1:\cdots:x_k]^T$ where $(x_3,\ldots,x_k)\neq(0,\ldots,0)$. Moreover, $(u,z,w)\in [(u,v,x)]$ is equivalent to the existence of $B \in GL(k,n)$ such that $B(u,v,x)=(u,z,w)$. Such a $B$ is characterized by having nonzero multiples of $u'$, $z'$, and $w'$ as its first three columns, so $B$ exists provided these three are linearly independent. Similarly, $(z,v,w)\in[(u,v,x)]$ is equivalent to the existence of $A\in GL(k,n)$ such that $A(u,v,x)=(z,v,w)$. Such an $A$ exists if and only if $v'$, $z'$ and $w'$ are linearly independent. Thus, the $z$ satisfying the above inclusions are those of the form $z=[x_1:\cdots:x_k]^T$, where $(x_3,\ldots,x_k)\neq(0,\ldots,0)$. This yields $p_{ijk}^l=\frac{(n^{k-2}-1)n^2}{n-1}$ values for $z$ when $R_l=R^w$. Similarly, $p_{ijk}^l=\frac{n^k-1}{n-1}-3n$ when $R_l=R^x$. 
This proves the claim. The other intersection numbers are obtained similarly and are located in Table \ref{tab:pgaml}.

\begin{table}[htbp!]
    \centering
    \begin{tabular}{|c|}

\hline
 $A^w A^w A^w = (n-2)A^w$ \\ \hline
 $A^w A^w A^x =A^w A^x A^w = A^x A^w A^w =0$ \\ \hline
 $A^w A^x A^x = A^x A^w A^x =A^x A^x A^w =(n-1)A^x$ \\ \hline
 $A^x A^x A^x = \frac{(n^{k-2}-1)n^2}{n-1}A^w + (\frac{n^k-1}{n-1}-3n)A^x$   \\  \hline

$I_1 A^w A^w = (n-1) I_1$\\  \hline
 $ A^w I_2 A^w = (n-1) I_2$ \\ \hline
$ A^w A^w  I_3 = (n-1) I_3$\\ \hline
 $I_1 A^w A^x = I_1 A^x A^w = 0$\\ \hline
 $A^w I_2 A^x = A^x I_2 A^w = 0$\\ \hline
 $A^w A^x  I_3 =A^x A^w  I_3 = 0$ \\ \hline
$I_1 A^x A^x = \frac{n^2(n^{k-2}-1)}{n-1} I_1$ \\ \hline
 $ A^x I_2 A^x = \frac{n^2(n^{k-2}-1)}{n-1} I_2$\\ \hline
$ A^x A^x  I_3 = \frac{n^2(n^{k-2}-1)}{n-1} I_3$ \\\hline

    \end{tabular}

\medskip
    
    \caption{Intersection numbers of ASTs from $P\Gamma L(k,n)$, $k\geq 3$}
    \label{tab:pgaml}
\end{table}

\begin{rem} Let $k\geq 3$ and $n$ be a prime power.
\begin{enumerate}
    \item The ternary subalgebra spanned by the adjacency hypermatrices $A^w$ and $A^x$ of the AST from $P\Gamma L(k,n)$ is a commutative ternary algebra. Moreover, $A^w$ spans its own ternary subalgebra.
    \item Let $H$ be a subgroup of $P\Gamma L(k,n)$ containing $PSL(k,n)$. It can be verified that the orbits of $H_{u,v}$ are the same as the orbits of $P\Gamma L(k,n)_{u,v}$. By Lemma \ref{lem_mesner_orbit} and Remark \ref{rem_orbit_implies_equality}, the ASTs obtained from $H$ and $P\Gamma L(k,n)$ are equal. 
\end{enumerate}

\end{rem}

\subsection{\texorpdfstring{${PSL(2,n)}$}{PSL(2,n)}}
Let $n$ be an an odd prime power and fix a quadratic nonresidue $\eta \in GF(n)$. Further, take $u=[1:0]^T$, $v=[0:1]^T$, $w=[1:1]^T$, and $s=[\eta:1]^T$ from $PG(1,n)$. Lemma 4.3 of \cite{mesner_association_1990} states that the two-point stabilizer $PSL(k,n)_{u,v}$ has two equally-sized orbits. These are $\{[a:1]^T:a \text{ a quadratic non-residue}\}$ and $\{[a:1]^T:a \text{ a quadratic residue, } a\neq 0,1\}$.

It follows that the AST obtained from $PSL(2,n)$ has two nontrivial relations $R^w=[(u,v,w)]$ and $R^s=[(u,v,s)]$. Letting $A^w$ and $A^s$ be the adjacency hypermatrices corresponding to $R^w$ and $R^s$, we complete the list in \cite[Theorem 4.6]{mesner_association_1990} of this AST's intersection numbers and tabulate them in Table \ref{tab:psl}. We illustrate only that $A^w A^s A^s = \frac{n-1}{4}A^w$ when $n\equiv 1 \pmod 4$ as computations of the other identities are similar. As with the prior subsection, it suffices to determine $p_{ijk}^l$ when $R_i=R^w$, $R_j=R_k=R^s$, and $R_l\in \{A^w,A^s\}$. Taking $R_l=R^w$ and $(u,v,w)\in R_l$, it can be deduced from Lemma 4.3 of \cite{mesner_association_1990} that satisfying the inclusions $(z,v,w)\in R^w$, $(u,z,w)\in R^s$, and $(u,v,z)\in R^s$ is equivalent to $z=[c:1]^T$, where $c$ and $1-c$ are quadratic non-residues. The number $p_{ijk}^l$ of such $c$ is the number of $x^2\in GF(n)$ such that there is a $y^2$ with $\eta x^2+ \eta y^2=1$. This equation always has a solution $(a_0,b_0)$ in $GF(n)$, since the sets $\lr{\eta a^2:a\in GF(n)}$ and $\lr{1- \eta a^2:a\in GF(n)}$ each have $\frac{n+1}{2}$ elements and so must intersect. The solutions $(x,y)$ are then given by $(a_0,b_0) + (a,b)t_{a,b}$ where $t_{a,b}=-2\frac{\eta a a_0+\eta b b_0}{\eta a^2+\eta b^2}$ and $(0,0)\neq(a,b)$. Notice that the obtained solution from $(a,b)$ is the same as the obtained solutions from any nonzero multiple of $(a,b)$. Further, the parametrization fails if and only if $\eta a^2 = -\eta b^2$; i.e., if $\frac{a}{b}$ is a square root of $-1$. This yields a total of $n-1$ solutions $(x,y)$ to $\eta x^2+ \eta y^2=1$. Since we only desire the number of $x^2$, this yields $p_{ijk}^l=\frac{n-1}{4}$. Similarly, $p_{ijk}^l=0$ when $R_l=R^s$, completing the verification.

\begin{table}[htbp!]
    \centering
    \begin{tabular}{|c|}

\hline
$A^w A^w A^w = 
    \begin{cases} 
    \frac{n-5}{4} A^w, & n\equiv 1\pmod{4}\\
    \frac{n+1}{4} A^s, &n\equiv 3\pmod{4}.
    \end{cases}$ \\\hline
    
    $A^w A^w A^s = A^w A^s A^w = A^s A^w A^w = \begin{cases} 
    \frac{n-1}{4} A^s, & n\equiv 1\pmod{4}\\
    \frac{n-3}{4} A^w, &n\equiv 3\pmod{4}.
    \end{cases}$ \\\hline

  $A^w A^s A^s = A^s A^w A^s = A^s A^s A^w = \begin{cases} 
    \frac{n-1}{4} A^w, & n\equiv 1\pmod{4}\\
    \frac{n-3}{4} A^s, &n\equiv 3\pmod{4}.
    \end{cases}$ \\\hline
    
    $A^s A^s A^s = 
    \begin{cases} 
    \frac{n-5}{4} A^s, & n\equiv 1\pmod{4}\\
    \frac{n+1}{4} A^w, &n\equiv 3\pmod{4}.
    \end{cases}$   \\\hline

$A_1 A^w A^w = A_1 A^s A^s =
    \begin{cases} 
    \frac{n-1}{2} A_1, & n\equiv 1\pmod{4}\\
    0, &n\equiv 3\pmod{4}.
    \end{cases}$ \\\hline
    
    $A^w A_2 A^w = A^s A_2 A^s=
    \begin{cases} 
    \frac{n-1}{2} A_2, & n\equiv 1\pmod{4}\\
    0, &n\equiv 3\pmod{4}.
    \end{cases}$\\\hline
    
$A^w A^w A_3 = A^s A^s A_3 =
    \begin{cases} 
    \frac{n-1}{2} A_3, & n\equiv 1\pmod{4}\\
    0, &n\equiv 3\pmod{4}.
    \end{cases}$\\\hline
    
    $A_1 A^w A^s = A_1 A^s A^w =
    \begin{cases} 
    0, & n\equiv 1\pmod{4}\\
    \frac{n-1}{2} A_1, &n\equiv 3\pmod{4}.
    \end{cases}$\\\hline
$A^w A_2 A^s = A^s A_2 A^w =
    \begin{cases} 
    0, & n\equiv 1\pmod{4}\\
    \frac{n-1}{2} A_2, &n\equiv 3\pmod{4}.
    \end{cases}$\\\hline
    
    $A^w A^s A_3 = A^s A^w A_3 =
    \begin{cases} 
    0, & n\equiv 1\pmod{4}\\
    \frac{n-1}{2} A_3, &n\equiv 3\pmod{4}.
    \end{cases}$\\ \hline

    \end{tabular}

\medskip
    
    \caption{Intersection numbers of ASTs from $PSL(2,n)$, $n$ odd}
    \label{tab:psl}
\end{table}

\begin{rem}
    Let $n$ be an odd prime power.
    \begin{enumerate}
        \item The ternary subalgebra spanned by the adjacency matrices $A^w$ and $A^s$ of the AST from $PS L(2,n)$ is a commutative ternary algebra. Moreover, $A^w$ and $A^s$ each span their own ternary subalgebra when $n\equiv 1 \pmod 4$.
        \item From Table \ref{tab:psl}, the function interchanging $A^w$ and $A^s$ extends linearly to a ternary algebra automorphism of the ternary subalgebra spanned by $A^w$ and $A^s$.   
        \item In Theorem 4.6 (ii) of \cite{mesner_association_1990}, it is claimed that the AST obtained from $PSL(2,n)$ satisfies the following. \[A^w A^s A_1 = \begin{cases}
    0, & n\equiv 1\pmod{4}\\
    \frac{n-1}{2} A_1, &n\equiv 3\pmod{4}. \end{cases}\]
This cannot be the case, as Lemma \ref{rem_trivial_prod} assures that $A^w A^s A_1$ is always 0.
    \end{enumerate}
\end{rem}

\subsection{\texorpdfstring{${A\Gamma L(k,n)}$}{AGamL(k,n)}}


Let $k\geq2$, $n=p^\alpha$ be a prime power, and $H=Gal(GF(n))$. Let $t=(0,1,0,\ldots,0)^T$ and $\Vec{a}=(a,0,\ldots,0)^T$ for each $a \in GF(n)$ be the given elements from the $k$-dimensional vector space over $GF(n)$. The elements of the two-point stabilizer $A\Gamma L(k,n)_{\vec{0},\vec{1}}$ are the transformations of the form $x\mapsto A\phi(x)$ for some $\phi \in H$ and $A\in GL(k,n)$. Applying these to the elements of the affine space yields two types of orbits. The orbits of the first type have the form $\{{\vvec{\phi(a)}:\phi \in H}\}$ where $a\neq 0,1$. The orbits of this type are in correspondence with the Galois conjugacy classes of $H$ and have corresponding sizes $\operatorname{deg}_{GF(p)}(a)$. Letting $\kappa$ be the number of orbits of the first type and letting $H$ act on $GF(n)$, the Burnside Orbit Counting Theorem implies 
\[\kappa+2 = \frac{1}{\vert H\vert }\sum_{g \in H} \vert Fix(g)\vert  
= \frac{1}{\alpha} \sum_{\beta=1}^{{\alpha}} \vert Fix(x\mapsto x^{p^\beta})\vert
= \frac{1}{\alpha}\sum_{\beta=1}^{\alpha} p^{\gcd({\alpha},\beta)}.\] 
The addend 2 is due to $0$ and $1$ in $GF(n)$, which are stabilized by field automorphisms.
The only orbit of the remaining type has size $n^k-n$ and consists of the vectors linearly independent from $\vec{1}$.

Let $X$ be the AST obtained from the group $A\Gamma L(k,n)$. For each $a\in GF(n)$, denote by $R^a$ the relation in $X$ containing the triple $(\vec{0},\vec{1},\vec{a})$ and let $A^a$ be the corresponding adjacency hypermatrix. Further, let $R^*$ denote the relation $[(\vec{0},\vec{1},t)]$ and let $A^*$ be the corresponding adjacency hypermatrix. Finally, fix any transversal $T$ of the orbits of $H$ on $GF(n)\setminus\lr{0,1}$ so that $\lr{R^a:a\in T}$ is the set of nontrivial relations of $X$ of the form $[(\vec{0},\vec{1},\vec{a})]$ for $a\neq 0,1$.

We list the intersection numbers of $X$ in Table \ref{tab:agaml}. To verify the intersection numbers, take $a,b,c\notin\{0,1\}$. We prove $A^a A^b A^c = \sum_{\ell\in T} p_{\ell} A^\ell$, where \begin{equation}\label{eqn:agaml}
    p_{\ell} = \vert \lr{\phi(c) : \phi \in H, \text{ and } (\exists \psi,\tau \in H)((1-\phi(c))\tau(a) + \phi(c) = \psi(b)\phi(c) = \ell)}\vert  .
\end{equation}

Indeed, let $R_i=R^a$, $R_j=R^b$, $R_k=R^c$ and $R^l\in \{R_a:a\in T\}\cup\{R^*\}$. We first locate $R_l=[(\vec{0},\vec{1},v)]$ such that $p_{ijk}^l$ may be nonzero. Take $(\vec{0},\vec{1},v)\in R_l$. Direct computation reveals that $(\vec{0},\vec{1},z)\in [(\vec{0},\vec{1},\vec{c})]$ is equivalent to $z=\vvec{\phi(c)}$ for some $\phi\in H$. Now, $(\vec{0},z,v)\in (\vec{0},\vec{1},\vec{b})$ is equivalent to the existence of $B\in GL(k,n)$ and $\psi \in H$ such that $B(\vec{0},\vec{1},\vvec{\psi(b)})=(\vec{0},z,v)$. Such $B\in GL(k,n)$ are characterized as those whose first column is $z$, with $v=B\vvec{\psi(b)}=\psi(b)z$. 
Furthermore, $(z,\vec{1},v) \in [(\vec{0},\vec{1},\vec{a})]$ is equivalent to the existence of $A\in GL(k,n)$ and $\tau\in H$ such that $A(\vec{0},\vec{1},\vvec{\tau(a)}) + (z,z,z)= (z,\vec{1},v)$. Such $A$ are characterized by those whose first column is $\vec{1}-z$, with $v=\tau(a)(\vec{1}-z) + z $. These three conditions necessitate that $v=\tau(a)(\vec{1}-\vvec{\phi(c)}) + \vvec{\phi(c)} = \vvec{\psi(b)\phi(c) }$ for some $\phi,\psi,\tau\in H$. Hence, for $p_{ijk}^l$ to be nonzero, $R_l$ must be of the form $R^\ell=[(\vec{0},\vec{1},\vec{\ell})]$, where $ \psi(b) \phi(c) = \ell = \tau(a)(1-\phi(c))+\phi(c)$ for some $\phi,\psi,\tau\in H$. Fix such an $\ell$ via fixing such $\phi,\psi,\tau\in H$. 
With $R_l=R^\ell$, repeating the above reasoning yields Equation (\ref{eqn:agaml}).
All other intersection numbers are obtained similarly.

\begin{table}[htbp!]
    \centering
    \begin{tabular}{|c|}

\hline
$A^a A^b A^c = \sum_{\ell\in T} p_{\ell} A^\ell$\\ \hline

$A^a A^b A^* = A^a A^* A^b  =A^*A^a A^b =0$\\ \hline

$A^a A^* A^* =A^* A^a A^*=A^* A^* A^a= p_* A^*$\\ \hline

$A^* A^* A^* = (n^k-3n+3)A^* + \sum_{a\in T} (n^k-n)A^a $ \\ \hline

$I_1 A^a A^b = p_1 I_1$\\ \hline

$A^a I_2 A^b = p_2 I_2$\\  \hline

$ A^a A^b I_3 = p_3 I_3$\\ \hline

$I_1 A^a A^* = I_1 A^* A^a  = 0$\\ \hline
$A^a I_2 A^* = A^* I_2 A^a = 0 $\\ \hline 
$A^a A^* I_3 = A^* A^a I_3 =0$ \\ \hline
$I_1 A^* A^* = (n^k-n)I_1$\\ \hline
$A^* I_2 A^* = (n^k-n)I_2$\\ \hline
$A^* A^* I_3 = (n^k-n)I_3$ \\ \hline
    \end{tabular}
\smallskip
\begin{tablenotes}
        \item $ p_{\ell} = \vert \lr{\phi(c) : \phi \in H \text{ and } (\exists \psi,\tau \in H)((1-\phi(c))\tau(a) + \phi(c) = \psi(b)\phi(c) = \ell)}\vert $
        \item $p_*=\vert\{\tau(a):\tau \in H\} \vert=\operatorname{deg}_{\operatorname{Fix}(H)}(a)$
        \item $p_1 = \vert \{\psi(b) : \psi \in H \text{ and } (\exists \tau\in H)(\tau(a) \psi(b)=1)\}\vert$
        \item $p_2 = \vert \{\psi(b) : \psi\in H \text{ and } (\exists \tau\in H)(\tau(a)\psi(b)=\tau(a)+ \psi(b) \}\vert$ 
        \item $p_3 = \vert \{\psi(b) : \psi\in H \text{ and }(\exists \tau\in H)(\tau(a) + \psi(b) =1)\}\vert$
    \end{tablenotes}
    
\medskip 
    \caption{Intersection numbers of ASTs from $A\Gamma L(k,n)$, $k\geq 2$}
    \label{tab:agaml}
\end{table}
The next remark extends our results to other subgroups of $A\Gamma L(k,n)$ and the case where $k=1$.

\begin{rem}\label{rem:agaml}Let $k\geq 1$ and $n=p^\alpha$ be a prime power.
\begin{enumerate} 
    \item The adjacency hypermatrices of the form $A^a$ with $a\in T$ span a ternary algebra.
    \item When $k=1$, the orbits of $A\Gamma L(1,n)_{0,1}$ will only be of the first type. In particular, there is no $R^*$ but the third valencies and intersection numbers not involving $R^*$ remain the same.
    \item The results extend to subgroups of $A\Gamma L(k,n)$ of the form \(AGL(k,n)\rtimes K\), where $K$ is any subgroup of $Gal(GF(n))$. In this case, $(AGL(k,n)\rtimes K)_{\vec{0},\vec{1}}$ retains two types of orbits. Orbits of the first type are of the form $\{{\vvec{\phi(a)}:\phi \in K}\}$ with respective sizes $\deg_{\operatorname{Fix}(K)}(a)$. These are in correspondence with the Galois conjugacy classes of $K$. Similar computations reveal that there are \(\kappa =-2+ \frac{\mathfrak{a}}{\alpha} \sum_{\beta=1}^{\frac{\alpha}{\mathfrak{a}}} (p^\mathfrak{a})^{\gcd{(\frac{\alpha}{\mathfrak{a}},\beta)}}\) orbits of the first type, where $p^\mathfrak{a}$ is the size of the fixed field $\operatorname{Fix}(K)$ of $K$. The set of vectors linearly independent from $\vec{1}$ remains the only orbit of the second type. Analogous computations verify that the intersection numbers of the AST from $AGL(k,n)\rtimes K$ are the same as those given in Table \ref{tab:agaml}, provided any mention of $H$ is replaced by $K$. 
\end{enumerate}
\end{rem}

As a particular case of Remark \ref{rem:agaml}.3, taking $K$ to be the trivial automorphism group yields the parameters of the AST obtained from $AGL(k,n)$. Further, taking $k=1$ corrects an error from \cite{mesner_association_1990}.

\begin{rem}
In Proposition 4.7 of \cite{mesner_association_1990}, it is claimed that the adjacency hypermatrices of the AST from $AGL(1,n)$ satisfy $A^b I_2 A^c = I_2$ and $A^b A^c I_3=I_3$ if $bc=1$, and that both products are 0 if $bc\neq 1$. However, these are incorrect. Indeed, if we consider the AST obtained from $AGL(1,5)$, we obtain $A^3 I_2 A^4 = I_2$ and $A^3 A^3 I_3 = I_3$. The correct statements are given below.

\begin{center}
    $A^a I_2 A^b = \begin{cases} I_2, &\text{if }ab=a+b, \\ 0, &\text{otherwise.}\end{cases}$ \;and\; $A^a A^b I_3 = \begin{cases} I_3, &\text{if }a+b=1,  \\ 0, &\text{otherwise.}\end{cases}$
\end{center}
\end{rem}

\subsection{\texorpdfstring{${Sz(2^{2k+1})}$}{Sz(2 2k+1)} 
}

Let $n=2^{2k+1}$ for some $k\geq 0$. In \cite[Proposition 4.8]{mesner_association_1990}, the authors determined that the size of the AST obtained from the Suzuki group $Sz(n)$ is $n+5$. The proof was not given completely, so we include a proof here, thereby also obtaining the third valencies of the nontrivial relations. 

Let $\sigma$ be the automorphism $a\mapsto a^{2^{k+1}}$ of $GF(n)$, $f:(GF(n))^2 \rightarrow GF(n)$ be the function $(x,y)\mapsto f_{xy}=xy + \sigma(x)x^2+\sigma(y)$, and $\Omega$ be the set \[\Omega= \lr{(x,y,f(x,y)):x,y \in GF(n)}\cup \lr{\infty}.\]
From \cite{dixon1996permutation}, $\Omega$ has $n^2+1$ points and $Sz(n)$ acts two-transitively on $\Omega$. By the same reference, the stabilizer of $(0,0,0)$ and $\infty$ is $Sz(n)_{0,\infty}=\lr{n_a : a \in GF(n)\setminus\lr{ 0} }$, where $n_a:\Omega \rightarrow \Omega$ is defined by \[n_a:(x,y,z) \mapsto (ax,\sigma(a)ay,\sigma(a)a^2z).\]

We show that there are two types of orbits when the two-point stabilizer is applied to $\Omega\setminus \{0,\infty\}$. Orbits of the first type have representatives with a nonzero first coordinate while the orbit of the second type will have representatives with a zero first coordinate. 

To see these, take any $(x,y,f_{xy})\in\Omega$ with $x\neq 0$. Applying the elements $n_a$ of $Sz(n)_{0,\infty}$ to $(x,y,f_{xy})$ yields $n-1$ distinct elements as $a$ ranges over the nonzero elements of $GF(n)$. In particular, exactly one element of the orbit of $(x,y,f_{xy})$ has a first coordinate of 1. By fixing $x=1$ and varying $y$, we obtain $n$ orbits of the first type. If we instead consider an element $(0,y,f_{0,y})$ with $y\neq 0$, similar reasoning gives $n-1$ elements in its orbit and a unique representative of this orbit with second coordinate 1. However, this orbit must then contain all elements of $\Omega \setminus \{0,\infty\}$ with first coordinate 0; hence, there is only one orbit of the second type.

\subsection{\texorpdfstring{${Ree(3^{2k+1})}$}{Ree(3 2k+1)}}
Let $n = 3^{2k+1}$ for some $k\geq 0$, $\sigma$ be the automorphism of $GF(n)$ given by $a \mapsto a^{3^{k+1}}$, and $f,g,h:(GF(n))^3 \rightarrow GF(n)$ be the functions given by 
\begin{align*}
f:(x,y,z)\mapsto f_{xyz}&=x^2y-xz+\sigma(y)-\sigma(x)x^3, \\
g:(x,y,z)\mapsto g_{xyz}&= \sigma(x)\sigma(y)-\sigma(z)+xy^2+yz-(\sigma(x))^2x^3,\\
h:(x,y,z)\mapsto h_{xyz}&=x\sigma(z)-\sigma(x)xy + \sigma(x)x^3y + x^2 y^2 -\sigma(y)y -z^2 + (\sigma(x))^2 x^4. 
\end{align*}
The group $Ree(n)$ acts two-transitively on the set \[\Omega = \lr{(x,y,z,f_{xyz},g_{xyz},h_{xyz}): x,y,z \in GF(n) } \cup \lr{\infty}\]
  of $n^3+1$ points \cite{dixon1996permutation}. By the same reference, the stabilizer of $(0,0,0,0,0,0)$ and $\infty$ is $Ree(n)_{0,\infty}=\lr{n_a : a \in GF(n)\setminus\lr{ 0} }$, where $n_a:\Omega \rightarrow \Omega$ is defined by \[n_a:(x,y,z,r,s,t) \mapsto (ax,\sigma(a)ay,\sigma(a)a^2z,\sigma(a)a^3 r,(\sigma(a))^2 a^3 s,(\sigma(a))^2 a^4 t).\]

  Determining the orbits of $Ree(n)_{0,\infty}$ is approached similarly to the determination of the orbits of $Sz(n)_{0,\infty}$. This yields three types of orbits. There are $n^2$ orbits of the first kind, each with a unique representative of the form $(1,y,z,f_{1yz},g_{1yz},h_{1yz})$. There are $n$ orbits of the second kind, each with a unique representative of the form $(0,y,1,f_{0y1},g_{0y1},h_{0y1})$. An orbit of either of these two types has size $n-1$. To obtain the remaining orbits, we reason as follows. 

  Take $y\neq 0$ and let $w = (0,y,0,\sigma(y),0,-\sigma(y)y)$. Suppose for $a,b\in GF(n)\setminus\lr{0}$ that \begin{align*}
    n_a(w)&=(0,\sigma(a)ay,0,\sigma(a)a^3\sigma(y),0,(\sigma(a))^2 a^4(-\sigma(y)y))\\&=(0,\sigma(b)by,0,\sigma(b)b^3\sigma(y),0,(\sigma(b))^2 b^4(-\sigma(y)y))=n_b(w).
\end{align*}

Using the middle equality and the second and fourth coordinates of the involved sextuples, we see that $\sigma(a)a =\sigma(b)b$, and $\sigma(a)a^3 = \sigma(b)b^3$. This tells us that $a=\pm b$. 
Hence, $n_a(w)$ takes $\frac{n-1}{2}$ distinct values as $a$ ranges over the nonzero elements of $GF(n)$; that is, the orbit of $w$ under $Ree(n)_{0,\infty}$ has $\frac{n-1}{2}$ elements. Since there are $n-1$
 sextuples of the form $(0,y,0,\sigma(y),0,-\sigma(y)y)$ in $\Omega$, there are two orbits of the third type. Considering possible values of $y$ reveals that $y$ either is or is not of the form $a\sigma(a)=a^{k+2}$ for some $a\neq 0$. Hence, taking $y=1$ and $y=b$, where $b$ is not a $(k+2)$-th power in $GF(n)$, yields the two orbits of the third type.  

\subsection{\texorpdfstring{${PGU(3,n)}$}{PGU(3,n)}} For $n$ a prime power, $PGU(3,n)$ is the group $GU(3,n)$ modulo its center. Here $GU(3,n)$ is the group of $3\times 3$ invertible matrices over $GF(n^2)$ which preserve the Hermitian form \[\varphi((u_1,u_2,u_3)^T,(v_1,v_2,v_3)^T)=u_1 {v_3}^n + u_2 {v_2}^n + u_3 v_1^n \] on the three-dimensional vector space over $GF(n^2)$. The group $PGU(3,n)$ acts two-transitively on \[\Omega = \lr{\lrr{(1,0,0)^T}} \cup \lr{\lrr{(a,b,1)^T}:a+{a}^n+b{b}^n=0,\;a,b\in GF(n^2)},\] the set of $n^3+1$ $\varphi$-isotropic lines \cite{dixon1996permutation}. Letting $E_1=\lrr{(1,0,0)^T}$ and $E_3=\lrr{(0,0,1)^T}$, \cite{dixon1996permutation} states that 
\[GU(3,n)_{E_1,E_3}=%
\{
Diag(c,d,c^{-n}):c,d\in GF(n^2), dd^n=1, c\neq 0
\}.
\]
By scaling each matrix in this set with the appropriate $\frac{1}{d}$, we see that $PGU(3,n)_{E_1,E_2}$ is a cyclic group isomorphic to $C_{n^2-1}$, each of whose elements is uniquely represented by a diagonal matrix $Diag(c,1,{c^{-n}})$ with $c\neq 0$. 
Given $\lrr{(a,b,1)^T}\in\Omega$ with $a,b\neq 0$, the elements of its orbit under $PGU(3,n)_{E_1,E_2}$ are $\lrr {(ac^{n+1},bc^n,1)^T}$, where $c\neq{0}$. Since $x\mapsto x^n$ is a field automorphism, this orbit has size $n^2-1$ and contains a unique representative of the form $\lrr{(r,1,1)^T}$ with $r+r^n+1=0$. On the other hand, the elements of the orbit of $\lrr{(a,0,1)^T}\in\Omega$, $a\neq 0$, are $\lrr{(ac^{n+1},0,1)^T}$, $c\neq{0}$. The function $f_{n+1}:x \mapsto x^{n+1}$ is an endomorphism of $GF(n^2)\setminus\lr{0}\cong C_{n^2-1}$ whose kernel consists of the $c\neq0$ whose order divides $n+1$. This is the multiplicative subgroup of the field isomorphic to $C_{gcd(n+1,n^2-1)}=C_{n+1}$. It follows that $Im(f_{n+1})\cong C_{n-1}$, so the orbit of $\lrr{(a,0,1)^T}\in\Omega$ has $n-1$ elements.

We now determine how many of each type of orbit exists. First, we determine how many elements of $\Omega$ are of the form $\lrr{(s,0,1)^T}$ for some $s\neq 0$. To be isotropic subspaces, $\lrr{(s,0,1)^T}$ with $s\neq 0$ must satisfy $s+s^n=0$ and $s\neq 0$. If $n$ were even, this condition is equivalent to $s^{n-1}=1$. These $s$ form a subgroup of the multiplicative group of the field isomorphic to $C_{gcd(n-1,n^2-1)}=C_{n-1}$. If $n$ were odd, the conditions $s+s^n=0$ and $s\neq 0$ are equivalent to $s^{n-1}=-1$. By reasoning as above, the set of $s$ that satisfy this condition is $Ker(f_{2(n-1)}) \setminus Ker(f_{n-1})$ where $f_{n-1}:x \mapsto x^{n-1}$ and $f_{2(n-1)}:x \mapsto x^{2(n-1)}$ are the given endomorphisms of $GF(n)^2\setminus\lr{0}$.  Thus, there are $n-1$ elements of $\Omega$ of the form $\lrr{(s,0,1)^T}$ for some $s\neq 0$, regardless of whether or not $n$ is even. Since an orbit of the second type has $n-1$ elements, there is only one such orbit. The remaining orbits must be of the first type, of which there are $\frac{n^3-1 - (n-1)}{n^2-1}=n$.

\begin{rem}
Recall that $PGU(3,n)=PSU(3,n)$ whenever $3$ does not divide $n+1$ and that $PSU(3,n)$ is a proper subgroup of $PGU(3,n)$ otherwise \cite{dixon1996permutation}. Reasoning as above, each element of $PSU(3,n)_{E_1,E_3}$ is uniquely represented by a matrix of the form $Diag(c,1,c^{-n})$, where $c$ is in the order $\frac{n^2-1}{3}$ subgroup of $GF(n^2)\setminus\{0\}$. Similar computations then reveal that orbits of $PGU(3,n)_{E_1,E_2}$ of the form $[\langle(r,1,1)^T\rangle]$ will be a union of three equally sized orbits of $PSU(3,n)_{E_1,E_2}$. These three have representatives $\lrr{(a,1,1)^T}$, $\lrr{(b,\alpha,1)^T}$, and $\lrr{(c,\alpha^2,1)^T}$ for some $a,b,c\in GF(q^2)$. Here, $\{1,\alpha,\alpha^2\}$ is a full set of coset representatives of the order $\frac{n^2-1}{3}$ subgroup of $GF(n^2)\setminus\{0\}$. Meanwhile, the orbit of $\langle(s,0,1)^T\rangle$ in $PGU(3,n)_{E_1,E_2}$ remains unchanged as an orbit of $PSU(3,n)_{E_1,E_2}$.     
\end{rem}

\subsection{\texorpdfstring{${{Sp({2k},2)}}$}{Sp(2k,2)}}

The following description of the symplectic group uses \cite{dixon1996permutation}, \cite{kleidman1990subgroup}, and \cite{sastry2002doubly} as references. For $k\geq2$, the symplectic group $Sp(2k,2)$ is the subgroup of $GL(2k,2)$ preserving the bilinear form  \[b:((x_1,\ldots,x_k,y_1,\ldots,y_k)^T,(u_1,\ldots,u_k,v_1,\ldots,v_k)^T)\mapsto \sum_{i=1}^k (x_i v_i + u_i y_i)\] on the $2k$-dimensional vector space $V$ over $GF(2)$.
The group $Sp(2k,2)$ acts on the set $\Omega$ of quadratic forms $q$ on $V$ which satisfy $b(v,u)=q(v+u)-q(v)-q(u)$ for all $u,v\in V$. This action is given by $A q(v)=q(A^{-1}v)$ for $q\in \Omega$, $v\in V$, and $A\in Sp(2k,2)$. It is known that $Sp(2k,2)$ has two orbits $\Omega^+$ and $\Omega^-$ on $\Omega$ characterized by the Witt index. Those quadratic forms with Witt index $k$ belong to $\Omega^+$ while those with Witt index $k-1$ belong to $\Omega^-$. 
The actions of $Sp(2k,2)$ on $\Omega^+$ and $\Omega^-$ are the two-transitive actions of interest. We use the respective notations $Sp^+(2k,2)$ and $Sp^-(2k,2)$ when referring to the symplectic group with respect to these actions. Fix an $\varepsilon \in \lr{+,-}$ and let $q\in \Omega^\varepsilon$. The point stabilizer of a quadratic form $q$ in $Sp^\varepsilon(2k,2)$ is the orthogonal group $O^\varepsilon(q)$. 
Notice that $O^\varepsilon(q)$ acts naturally on the $2k$-dimensional vector space $V$ over $GF(2)$. In fact, the bijection $\tau: (x\mapsto q(x)+b(x,v))\mapsto v$ from $\Omega$ to $V$ is an isomorphism of $O^\varepsilon(q)$-sets. Under $\tau$, $0\in V$ corresponds to $q\in \Omega^\varepsilon$ and the isotropic vectors of $V$ with respect to $q$ correspond to the other elements of $\Omega ^\varepsilon$.

We utilize $\tau$ to describe the orbits of a two-point stabilizer. Fix $k\geq 3$ and $\varepsilon\in\{+,-\}$. Take distinct quadratic forms $q_1$ and $q_2$ in $\Omega^\varepsilon$. View $\Omega^\varepsilon$ as the set of $q_1$-isotropic vectors in $V$ via the $O^\varepsilon(q_1)$-set isomorphism $\tau$ so that $\tau(q_1)=0$ and $\tau(q_2)=v$ for some $q_1$-isotropic $v\in V$. Through this identification, we view the action of $Sp^\varepsilon(2k,2)_{q_1,q_2}=O^\varepsilon(q_1)_{q_2}$ on $\Omega^\varepsilon \setminus\lr{q_1,q_2}$ as the action on the $q_1$-isotropic vectors of $V\setminus\lr{0,v}$. By Proposition 2 (ii) of \cite{king1981some}, one orbit consists of the nonzero $q_1$-isotropic vectors distinct from and orthogonal to $v$ while the other orbit consists of the nonzero $q_1$-isotropic vectors not orthogonal to $v$. Proposition 1 in \cite{Vasilev1995} then yields the respective sizes of these orbits, namely $ 2^{2k-2}+\varepsilon 2^{k-1}-2$ and $ 2^{2k-2}$.

\begin{rem}
    The above reasoning also applies to the case when $k=2$. Indeed, a two-point stabilizer of $Sp^+(4,2)$ will have two orbits, both of size $4$. However, $Sp^-(4,2)$ occurs as a degenerate case since $Sp^-(4,2)$ is three-transitive. In this case, any two-point stabilizer will have only one orbit, necessarily of size 4. 
\end{rem}

\subsection{Sporadics}
The sporadic two-transitive groups are the permutation representations of the Mathieu groups $M(n)$ of degree $n$ (with $n\in\lr{11,12,22,23,24}$), the projective group $PSL(2,11)$ of degree 11, the Mathieu group $M(11)$ of degree 12, the alternating group $A_7$ of degree 15, the Higman-Sims group $HS$ of degree 176, and the Conway group $Co_3$ of degree 276 \cite{dixon1996permutation}. Due to their higher transitivity, Example \ref{ex_int5} and Remark \ref{rem_oneNontrivial} apply to the ASTs from the Mathieu groups. As such, we refrain from tabulating their parameters. The AST from $HS$ is known to have seven relations \cite{mesner_association_1990}. Using GAP 4.11.1 \cite{GAP4}, we provide the sizes and the third valencies of the ASTs from the other sporadic two-transitive groups in Table \ref{tab:thirdvalencies_and_sizes}. For the ASTs with more than one nontrivial relation, we tabulate their intersection numbers in Tables \ref{tab:PSL}, \ref{tab:A7}, \ref{tab:HS}, and \ref{tab:CO}. In particular, the subalgebra generated by the adjacency hypermatrices of the nontrivial relations of the AST obtained from any sporadic two-transitive group is commutative. 

\begin{table}[!htbp]
    \centering
    \begin{tabular}{cccccc}

$p_{144}^1=3$&$p_{155}^1=6$&$p_{424}^2=3$&$p_{443}^3=3$&$p_{445}^5=1$&$p_{454}^5=1$\\
$p_{455}^4=2$&$p_{455}^5=1$&$p_{525}^2=6$&$p_{544}^5=1$&$p_{545}^4=2$&$p_{545}^5=1$\\
$p_{553}^3=6$&$p_{554}^4=2$&$p_{554}^5=1$&$p_{555}^4=2$&$p_{555}^5=2$

    \end{tabular}
    \caption{Nonzero intersection numbers $p_{ijk}^l$ of the AST from $PSL(2,11)\leq M(24)$ when at most one of $i,j,k$ is less than 4.}
    \label{tab:PSL}
\end{table}

\begin{table}[!htbp]
    \centering
    \begin{tabular}{cccccc}
$p_{144}^1=1$&$p_{155}^1=12$&$p_{424}^2=1$&$p_{443}^3=1$&$p_{455}^5=1$&$p_{525}^2=12$\\
$p_{545}^5=1$&$p_{553}^3=12$&$p_{554}^5=1$&$p_{555}^4=12$&$p_{555}^5=9$

    \end{tabular}
    \caption{Nonzero intersection numbers $p_{ijk}^l$ of the AST from $A_7\leq PSL(4,2)$ when at most one of $i,j,k$ is less than 4.}
    \label{tab:A7}
\end{table}

\begin{table}[!htbp]
    \centering
    \begin{tabular}{cccccccc}
      $p_{144}^1=72$&$p_{155}^1=90$&$p_{166}^1=12$&$p_{424}^2=72$& $p_{443}^3=72$ &$p_{444}^4=20$\\
      
      $p_{445}^5=32$& $p_{445}^6=30$ & $p_{446}^5=4$&$p_{446}^6=6$&$p_{454}^5=32$&$p_{454}^6=30$\\
      $p_{455}^4=40$&$p_{456}^4=5$&$p_{464}^5=4$&$p_{464}^6=6$& $p_{465}^4=5$ &$p_{466}^4=1$\\
      $p_{525}^2=90$&$p_{544}^5=32$&$p_{544}^6=30$&$p_{545}^4=40$&$p_{546}^4=5$&$p_{553}^3=90$\\
    $p_{554}^4=40$&$p_{555}^5=41$&$p_{555}^6=60$&$p_{556}^5=8$&$p_{564}^4=5$&$p_{565}^5=8$\\
    $p_{626}^2=12$&$p_{644}^5=4$&$p_{644}^6=6$&$p_{645}^4=5$&$p_{646}^4=1$&$p_{654}^4=5$\\
    $p_{655}^5=8$&$p_{663}^3=12$&$p_{664}^4=1$&$p_{666}^6=5$ 

    \end{tabular}
    \caption{Nonzero intersection numbers $p_{ijk}^l$ of the AST from $HS$ when at most one of $i,j,k$ is less than 4.}
    \label{tab:HS}
\end{table}

\setlength{\tabcolsep}{5pt}
\begin{table}[!htbp]
    \centering
    \begin{tabular}{cccccc}
    $p_{144}^1=162$&$p_{155}^1=112$&$p_{424}^2=162$&$p_{443}^3=162$&$p_{444}^4=105$&$p_{445}^5=81$\\
    $p_{454}^5=81$&$p_{455}^4=56$&$p_{525}^2=112$&$p_{544}^5=81$&$p_{545}^4=56$&$p_{553}^3=112$\\
    $p_{554}^4=56$&$p_{555}^5=30$      

    \end{tabular}
    \caption{Nonzero intersection numbers $p_{ijk}^l$ of the AST from $Co_3$ when at most one of $i,j,k$ is less than 4.}
    \label{tab:CO}
\end{table}




\section*{Declarations}
\subsection*{Competing Interests}
On behalf of all authors, the corresponding author states that there is no conflict of interest.

\subsection*{Data and Source Code Availability}
The datasets generated and/or analysed during the current study, along with the source codes for obtaining these, are available from the corresponding author on reasonable request.

\bibliographystyle{amsplain}
 \bibliography{ternassoc}





\end{document}